\documentclass{amsart}

\usepackage{amsfonts,amssymb,verbatim,amsmath,amsthm,latexsym,textcomp,amscd}
\usepackage{latexsym,amsfonts,amssymb,epsfig,verbatim}
\usepackage{amsmath,amsthm,amssymb,latexsym,graphics,textcomp}
\usepackage{graphicx}
\usepackage{color}
\usepackage{url}

\input xy
\xyoption{all}

\begin{document}

\newtheorem{theorem}{Theorem}[section]
\newtheorem{prop}[theorem]{Proposition}
\newtheorem{lemma}[theorem]{Lemma}
\newtheorem{sublemma}[theorem]{Sublemma}
\newtheorem{cor}[theorem]{Corollary}
\newtheorem{defn}[theorem]{Definition}
\newtheorem{conj}[theorem]{Conjecture}
\newtheorem{rmk}[theorem]{Remark}
\newtheorem{qn}[theorem]{Question}
\newtheorem{claim}[theorem]{Claim}
\newtheorem{defth}[theorem]{Definition-Theorem}

\newcommand{\boundary}{\partial}
\newcommand{\C}{{\mathbb C}}
\newcommand{\U}{{\mathbb U}}
\newcommand{\Hyp}{{\mathbb H}}
\newcommand{\D}{{\mathbb D}}
\newcommand{\Z}{{\mathbb{Z}}}
\newcommand{\R}{{\mathbb R}}
\newcommand{\Q}{{\mathbb Q}}
\newcommand{\E}{{\mathbb E}}
\newcommand{\ints}{{\mathbb{Z}}}
\newcommand{\natls}{{\mathbb N}}
\newcommand{\ratls}{{\mathbb Q}}
\newcommand{\rls}{{\mathbb R}}
\newcommand{\proj}{{\mathbb P}}
\newcommand{\lhp}{{\mathbb L}}
\newcommand{\tube}{{\mathbb T}}
\newcommand{\cusp}{{\mathbb P}}
\newcommand\A{{\mathbb A}}
\newcommand\BB{{\mathcal B}}
\newcommand\CC{{\mathcal C}}
\newcommand\DD{{\mathbb D}}
\newcommand\EE{{\mathcal E}}
\newcommand\FF{{\mathcal F}}
\newcommand\GG{{\mathcal G}}
\newcommand\HH{{\mathcal H}}
\newcommand\II{{\mathcal I}}
\newcommand\JJ{{\mathcal J}}
\newcommand\KK{{\mathcal K}}
\newcommand\LL{{\mathcal L}}
\newcommand\MM{{\mathcal M}}
\newcommand\NN{{\mathcal N}}
\newcommand\OO{{\mathcal O}}
\newcommand\PP{{\mathcal P}}
\newcommand\QQ{{\mathcal Q}}
\newcommand\RR{{\mathcal R}}
\newcommand\SM{{\mathbb A}}
\newcommand\Sp{{\mathbb S}}
\newcommand\TT{{\mathcal T}}
\newcommand\UU{{\mathcal U}}
\newcommand\VV{{\mathcal V}}
\newcommand\WW{{\mathcal W}}
\newcommand\XX{{\mathcal X}}
\newcommand\YY{{\mathcal Y}}
\newcommand\ZZ{{\mathcal{Z}}}
\newcommand\CH{{\CC\HH}}
\newcommand\PEY{{\PP\EE\YY}}
\newcommand\MF{{\MM\FF}}
\newcommand\RCT{{{\mathcal R}_{CT}}}
\newcommand\RCTT{{{\mathcal R}^2_{CT}}}
\newcommand\PMF{{\PP\kern-2pt\MM\FF}}
\newcommand\FL{{\FF\LL}}
\newcommand\PML{{\PP\kern-2pt\MM\LL}}
\newcommand\GL{{\GG\LL}}
\newcommand\Pol{{\mathcal P}}
\newcommand\half{{\textstyle{\frac12}}}
\newcommand\Half{{\frac12}}
\newcommand\Mod{\operatorname{Mod}}
\newcommand\Area{\operatorname{Area}}
\newcommand\ep{\epsilon}
\newcommand\hhat{\widehat}
\newcommand\Proj{{\mathbf P}}
\newcommand\til{\widetilde}
\newcommand\length{\operatorname{length}}
\newcommand\tr{\operatorname{tr}}
\newcommand\gesim{\succ}
\newcommand\lesim{\prec}
\newcommand\simle{\lesim}
\newcommand\simge{\gesim}
\newcommand{\simmult}{\asymp}
\newcommand{\simadd}{\mathrel{\overset{\text{\tiny $+$}}{\sim}}}
\newcommand{\ssm}{\setminus}
\newcommand{\diam}{\operatorname{diam}}
\newcommand{\pair}[1]{\langle #1\rangle}
\newcommand{\T}{{\mathbf T}}
\newcommand{\inj}{\operatorname{inj}}
\newcommand{\pleat}{\operatorname{\mathbf{pleat}}}
\newcommand{\short}{\operatorname{\mathbf{short}}}
\newcommand{\vertices}{\operatorname{vert}}
\newcommand{\collar}{\operatorname{\mathbf{collar}}}
\newcommand{\bcollar}{\operatorname{\overline{\mathbf{collar}}}}
\newcommand{\I}{{\mathbf I}}
\newcommand{\tprec}{\prec_t}
\newcommand{\fprec}{\prec_f}
\newcommand{\bprec}{\prec_b}
\newcommand{\pprec}{\prec_p}
\newcommand{\ppreceq}{\preceq_p}
\newcommand{\sprec}{\prec_s}
\newcommand{\cpreceq}{\preceq_c}
\newcommand{\cprec}{\prec_c}
\newcommand{\topprec}{\prec_{\rm top}}
\newcommand{\Topprec}{\prec_{\rm TOP}}
\newcommand{\fsub}{\mathrel{\scriptstyle\searrow}}
\newcommand{\bsub}{\mathrel{\scriptstyle\swarrow}}
\newcommand{\fsubd}{\mathrel{{\scriptstyle\searrow}\kern-1ex^d\kern0.5ex}}
\newcommand{\bsubd}{\mathrel{{\scriptstyle\swarrow}\kern-1.6ex^d\kern0.8ex}}
\newcommand{\fsubeq}{\mathrel{\raise-.7ex\hbox{$\overset{\searrow}{=}$}}}
\newcommand{\bsubeq}{\mathrel{\raise-.7ex\hbox{$\overset{\swarrow}{=}$}}}
\newcommand{\tw}{\operatorname{tw}}
\newcommand{\base}{\operatorname{base}}
\newcommand{\trans}{\operatorname{trans}}
\newcommand{\rest}{|_}
\newcommand{\bbar}{\overline}
\newcommand{\UML}{\operatorname{\UU\MM\LL}}
\newcommand{\EL}{\mathcal{EL}}
\newcommand{\tsum}{\sideset{}{'}\sum}
\newcommand{\tsh}[1]{\left\{\kern-.9ex\left\{#1\right\}\kern-.9ex\right\}}
\newcommand{\Tsh}[2]{\tsh{#2}_{#1}}
\newcommand{\qeq}{\mathrel{\approx}}
\newcommand{\Qeq}[1]{\mathrel{\approx_{#1}}}
\newcommand{\qle}{\lesssim}
\newcommand{\Qle}[1]{\mathrel{\lesssim_{#1}}}
\newcommand{\simp}{\operatorname{simp}}
\newcommand{\vsucc}{\operatorname{succ}}
\newcommand{\vpred}{\operatorname{pred}}
\newcommand\fhalf[1]{\overrightarrow {#1}}
\newcommand\bhalf[1]{\overleftarrow {#1}}
\newcommand\sleft{_{\text{left}}}
\newcommand\sright{_{\text{right}}}
\newcommand\sbtop{_{\text{top}}}
\newcommand\sbot{_{\text{bot}}}
\newcommand\sll{_{\mathbf l}}
\newcommand\srr{_{\mathbf r}}
\newcommand\geod{\operatorname{\mathbf g}}
\newcommand\mtorus[1]{\boundary U(#1)}
\newcommand\AF{\mathbf A}
\newcommand\Aleft[1]{\A\sleft(#1)}
\newcommand\Aright[1]{\A\sright(#1)}
\newcommand\Atop[1]{\A\sbtop(#1)}
\newcommand\Abot[1]{\A\sbot(#1)}
\newcommand\boundvert{{\boundary_{||}}}
\newcommand\storus[1]{U(#1)}
\newcommand\Momega{\omega_M}
\newcommand\nomega{\omega_\nu}
\newcommand\twist{\operatorname{tw}}
\newcommand\modl{M_\nu}
\newcommand\MT{{\mathbb T}}
\newcommand\Teich{{\mathcal T}}
\newcommand{\etalchar}[1]{$^{#1}$}
\newcommand{\noi}{\noindent}
\newcommand{\ra}{\rightarrow}
\newcommand{\ua}{\uparrow}
\newcommand{\lra}{\longrightarrow}
\newcommand{\dra}{\downarrow}
\newcommand{\bo}{\textbf}
\newcommand{\spa}{\vspace{.25in}}
\newcommand{\sing}{\mathit{Sing}}

\renewcommand{\Re}{\operatorname{Re}}
\renewcommand{\Im}{\operatorname{Im}}

\title[Nash Equilibria]{Nash Equilibria via Duality and Homological Selection}

\author[Arnab Basu]{Arnab Basu}

\address{Quantitative Methods and Information Systems Area,
Indian Institute of Management Bangalore-560076, India}

\email{arnab.basu@iimb.ernet.in}

\author[Samik Basu]{Samik Basu}

\address{RKM Vivekananda University, Belur Math, WB-711 202, India}

\email{samik.basu2@gmail.com; samik@rkmvu.ac.in}

\author[Mahan Mj]{Mahan Mj}

\address{RKM Vivekananda University, Belur Math, WB-711 202, India}

\email{mahan.mj@gmail.com; mahan@rkmvu.ac.in}



\date{\today}

\thanks{Research of  first author  supported partly by DST GoI Grant no: SR/FTP/MS-08/2009 and partly by IIM Bangalore Research Chair Grant no: 190. Research of  third author  partly supported by  CEFIPRA Indo-French Research Grant 4301-1. }

\begin{abstract} Given a multifunction from $X$ to the $k-$fold symmetric product
$Sym_k(X)$, we use the Dold-Thom Theorem to establish a homological  selection Theorem.
This is used to establish existence of Nash equilibria.
Cost functions  in  problems concerning the existence of
Nash Equilibria are traditionally multilinear in the mixed strategies. The main aim of this paper is to relax the hypothesis of 
 multilinearity.  We use  basic intersection theory, Poincar\'e Duality in addition to  the Dold-Thom Theorem.
\end{abstract}

\maketitle

\tableofcontents

\section{Introduction} The main topological problem addressed in this paper is the following: \\ 

Let $X$ be a metric space and $Sub_k(X)$ denote the collection of subsets 
of $X$ with at most $k$ points equipped with the Hausdorff metric.  Let $Sym_k(X)$ denote the $k-$fold symmetric product of $X$, also called the configuration space
of $k$ points (counted with multiplicity) in $X$.  Given a multifunction $s : X \rightarrow Y$, where $Y$ is  $Sub_k(X)$ or  $Sym_k(X)$,
 does there exist a "homological" (in a  sense to be made precise) selection? \\
A related problem addressed in the paper is: \\
Can one lift a map  $Z \rightarrow Sub_k(X)$ to a map   $Z \rightarrow Sym_k(X)$? 

This topological  problem is used to establish existence of Nash equilibria by specializing $X$ to a simplex. 
The property of ”admitting homological selection” (Definition \ref{homseldef}), in a similar context of establishing existence equilibria in certain games,
has been already considered under the name ”spanning property” (or ”property $\mathcal S$”) in \cite{spiez1}  (see
also \cite{spiez2, spiez3}).

\subsection{Preliminaries}
We  shall work in the framework of
non-cooperative games with mixed strategies, where the domain for every player is a finite dimensional
simplex. 
The  cost functions (or alternately payoff functions) in traditional problems of
Nash Equilibria are multilinear in the mixed strategies. The main aim of this paper is to relax the hypothesis that 
 cost functions  are multilinear. The original proof of existence of Nash Equilibria \cite{nash} uses fairly
 simple  Algebraic Topology, namely
 the Brouwer or Kakutani Fixed Point Theorems. This approach has been refined in various ways \cite{lh, gale, km}
(see also \cite{hs} for an effective approach using a minimax technique).
Our approach in this paper is quite different inasmuch as
we use  standard  but  more sophisticated tools  from Algebraic Topology
to establish  existence of Nash Equilibria  under considerably more general conditions on cost functions.
The tools we use are basic intersection theory, Poincar\'e Duality and the Dold-Thom Theorem. We use the 
 Dold-Thom Theorem to prove the existence of  certain relative cycles contained in graphs of multifunctions. 
The chains thus constructed may be thought of as homological versions of selections \cite{michael1, michael2, michael3}.
This furnishes us with a Homological Selection Theorem \ref{multisel}, which might be of independent interest.

The conditions we use on the cost functions are detailed in Section \ref{hyp}. We give a brief sketch here 
for a 2-player non-cooperative game to stress
the soft topological nature of the hypotheses used. Consider  a game between players 1 and 2 with mixed strategy spaces $\A_1, \A_2$.
We use the notation $\A_{-1}=\A_2$ and $\A_{-2}=A_1$.

A cost function $r_i(x,y)$ for player $i$ ($i=1,2$)
is said to be configurable if\\
1) the set of local minima of the best response multifunction $R_i(y)$ (for each $y\in \A_{-i}$)
of the function $r_i(x,y)$  ($x\in \A_{i}$) is finite; and \\
2) Counted with multiplicity, the set of local minima $R_i(y)$ is  continuous on $ \A_{-i}$.\\
A weakly configurable cost function is one whose set of local minima can be arbitrarily well-approximated by  
the set of local minima of a  configurable  function.

Two kinds of hypotheses will be relevant in this paper. \\
\noi \textbf{(A1)} For each $i$, the map $R_i$ is continuous.\\
\noi \textbf{(A2)} For each $i$, the map $R_i$ is weakly configurable.\\

The motivation for the assumption \textbf{(A2)} is given in Section \ref{hyp}.
The main Theorem of this paper (Theorem \ref{existNash}) proves that assumption \textbf{(A2)} is sufficient
to guarantee  the existence of Nash equilibria.
In Section \ref{ctreg} we shall give an example to show that assumption \textbf{(A1)} is not sufficient
to guarantee  the existence of Nash equilibria. We describe this counterexample in brief. It is easy to see that
the multifunction $f(z) = \pm \sqrt{z}$ from the unit disk $\Delta$ in the complex plane to itself has no continuous selection. However
the graph $gr(f)(\subset \Delta \times \Delta)$ does support a non-zero relative cycle in $H_2(\Delta \times \Delta, \partial \Delta \times \Delta)$
and hence admits a {\it homological selection} in our terminology (see Section \ref{hom}).
The counterexample in Section \ref{ctreg1}, which is a continuous map from $\DD^2$ to $Sub_3(\DD^2)$,
 shows that   a continuous multifunction need not admit
 even a  homological selection.  (Here $Sub_3(\DD^2)$ denotes the collection of subsets of $\DD^2$ with at most 3 points equipped with
the Hausdorff metric.) 

\subsection{Nash Equilibria} We refer to \cite{basar} for the basics of Game Theory.
 An $N$-person non-cooperative game is determined by $2N$ objects
$(\A_1,\ldots,\A_N,r_1,\ldots,r_N)$ where $\A_i$ denotes the
{\it strategy space} of player $i$ and
$r_i : \A \stackrel{def}{=} \A_1 \times \cdots \times \A_N \ra \R$ is the cost function
for player $i$. We shall call $\A$ the {\it total strategy space}.
Each $\A_i$ will, for the purposes of this paper,
 be the space of probability measures on a finite set of cardinality $(n_i+1)$
and hence homeomorphic to $\DD^{n_i}$. Thus, in Game Theory terminology each element
of $\A_i$ is a {\it mixed strategy} or equivalently a  probability measure on a finite set. Thus,  $\A_i$ is 
 the space of mixed strategies or player $i$. Vertices of the simplex  $\A_i$ are also referred to as {\it pure strategies}.
A mixed strategy may therefore be regarded as a {\it probability vector} with   $(n_i+1)$ components. 
 Each player $i$ independently chooses his strategy  $a_i \in \A_i$. For 
a strategy-tuple $a = (a_1, \cdots, a_N)$, player $i$
 pays an immediate cost $r_i(a)$. 

\smallskip

\begin{center}

 {\bf Goal:}
Each player wants to minimize his cost.

\end{center}

\bigskip

\noindent {\bf Notation:} We denote 
$\A_{-i} \stackrel{def}{=} \A_1 \times \cdots \A_{i-1} \times \A_{i+1} \times \A_N$ for all $i$.
For any strategy-tuple $(a_1, \cdots, a_N) \in \A$, $a_{-i} \stackrel{def}{=} (a_1,\ldots,
a_{i-1},a_{i+1},\ldots,a_N) \in \A_{-i}$. Also for convenience of notation we change the order of the variables
and assume that $r_i: \A_i \times \A_{-i}\ra \R$ is the i-th cost function.

\begin{defn} \label{Nash}
A strategy-tuple $a^* = (a_1^*,\ldots,a_N^*) \in \A$ is called a Nash equilibrium if
\begin{eqnarray}
&& {r}_i^* \stackrel{def}{=} {r}_i(a^*) \leq {r}_i(a_i,a_{-i}^*),\ \forall a_i \in \A_i,\ i =1,\ldots,N,
\end{eqnarray}
where, 
\end{defn}

Given a metric space $(X,d)$, let $\HH_c(X)$ denote the space of 
all compact subsets of $X$ equipped with the Hausdorff metric $d_H$ i.e, for all $A,B \in \HH_c(X)$,
\begin{equation}\label{hausmetric}
d_{H}(A,B) \stackrel{def}{=} \max\{\sup_{x\in A}\inf_{y \in B} d(x,y), \sup_{y \in B}\inf_{x \in A} d(x,y)\}.
\end{equation}

\begin{defn} \label{optresp}
For any player $i$, given a strategy-tuple $a_{-i} \in \A_{-i}$ of the other players, we define his best-response $R_i(a_{-i})$ as follows
\begin{equation}
R_i(a_{-i}) \stackrel{def}{=} {\arg\min}_{a_i \in \A_i} {r}_i(a_i,a_{-i}) \in \HH_c(\A_i).
\end{equation}
\end{defn}

Here ${\arg\min}$ denotes the {\it argument} of the minima set, i.e. the set of values $x\in \A_i$ such that
${r}_i(x,a_{-i})$ is minimum as  a function of $x$.

\smallskip

\noindent {\bf Note:} It will not affect our arguments at all if we take 
${\arg\min}$ to denote the  argument of the {\it local} minima set, i.e. the set of values $x\in \A_i$ such that
${r}_i(x,a_{-i})$ is a local minimum as  a function of $x$
provided we modify Definition \ref{Nash} appropriately. In such a case, we are looking at {\it local
Nash equilibria} which are defined by demanding that the ${r}_i^*$ appearing
in Definition \ref{Nash} are {\it local minima} rather than global minima as defined there.

\begin{defn} \label{optrespgr}
For any $i$, the graph $\mbox{gr}(R_i)$ of $R_i$ is defined as
\begin{equation}
\mbox{gr}(R_i) \stackrel{def}{=} \{(a_{-i},a_i): a_i \in R_i(a_{-i}),\  a_{-i} \in \A_{-i}\}.
\end{equation}
\end{defn}

\subsection{Motivational Example: Nash Equilibria for Bi-matrix Games}
We motivate our main Theorem by giving a simplified version of
our proof of the existence of Nash equilibria in the special case of standard 2-person bimatrix games. 

Consider a two-player bimatrix game where each player $i$ has strategy space $\A_i \stackrel{def}{=} \{(x_i,1-x_i): 0 \leq x_i \leq 1\}$, for $i=1, 2$.
 The cost function for player $i$ is a bilinear function of the form $r_i(x_1, x_2)=x_1^TM_ix_2$, where $M_1, M_2$ are $2 \times 2$ matrices and $x_1, x_2$
are probability vectors (mixed strategies) for players $1, 2$.

Then the best response $R_i(x_j)$ ($i \neq j$) of player $i$ when player $j$ chooses $x_j \in [0,1]$ is  given by
\begin{equation} \label{optrespeq}
R_i(x_j) = {\arg\min}_{x_i \in [0,1]} {r}_i(x_i,x_j),
\end{equation}
which is a {\it singleton} set $\{x_i^*(x_j)\}$ due to the component-wise linearity of the payoff function. Moreover,  the dependence of
$x_i^*$ on $x_j$ is continuous. Hence $\mbox{gr}(R_i)$ for $i=1, 2$ are continuous paths in $[0,1]\times [0,1]$
where  $\mbox{gr}(R_1)$ is a path from the bottom ($[0,1] \times \{ 0 \}$) to the top  ($[0,1] \times \{ 1 \}$) 
and   $\mbox{gr}(R_2)$  is a path from the left($\{ 0 \} \times [0,1]$) to the right  ($\{ 1 \} \times [0,1]$).
Hence $\mbox{gr}(R_1)$ and  $\mbox{gr}(R_2)$  must intersect (for a formal proof of this intuitively clear fact,
see  Lemma 2 of \cite{maehara} for instance). Hence
there exists $(x_1^*, x_2^*)$  such that $(x_1^* = R_1(x_2^*), x_2^* = R_2[x_1^*])$ i.e. a Nash equilibrium.

\subsection{Ordered and Unordered Tuples of Points}

Let $\DD^d \stackrel{def}{=} \{(x_1,\ldots,x_{d}) \in \R^{d}: \sum_{j=1}^{d} x_j^2 \leq 1\}$ denote the $d$-dimensional disk. Given a set $A$, let \#$A$ denote its cardinality. For positive integers $k$ and compact metric spaces $X$ we define  the topological space of subsets of cardinality $\leq k$,
$$Sub_k(X) \stackrel{def}{=} \lbrace subsets\, of \, X \, of \, cardinality\, \leq k \rbrace.$$
This has a topology induced as a subset of $\HH_c(X)$. An analogous construction is the configuration space,
\begin{equation}
C_k(X) \stackrel{def}{=} \{(x_1,\ldots,x_k) \in X^k\} / \Sigma_k
\end{equation}
where $\Sigma_k$ is the symmetric group on $k$ symbols. The elements of $C_k(X)$ can be written multiplicatively $x_1^{i_1}\ldots x_j^{i_j}$ with $i_1+\ldots +i_j =k$ or as equivalence classes $[x_1,\ldots , x_k]$. There is a metric on $C_k(X)$ given by 
$$d_k([x_1,\ldots , x_k] , [y_1,\ldots , y_k])= min_{\sigma \in \Sigma_k} ( sup_i (d(x_{\sigma (i)},y_i)) ).$$
We have a map $q_k: C_k(X) \rightarrow Sub_k(X)$ that maps $[x_1,\ldots , x_k]$ to $\lbrace x_1, \ldots , x_k \rbrace$. It  is a surjective continuous map (as $d_H (\lbrace x_1,\ldots , x_k \rbrace , \lbrace y_1, \ldots y_k \rbrace) \leq d_k([x_1,\ldots , x_k] , [y_1,\ldots , y_k])$). 

\begin{prop}
The quotient topology on $Sub_k(X)$ induced from the map $q_k$ is the same as the topology induced by the metric $d_H$. \label{subconf}
\end{prop}

\begin{proof}
The argument above already shows that if $U$ is an open set in the Hausdorff metric the inverse image is open and so it is open in the quotient topology. Conversely suppose $U$ is such that $q_k ^{-1}(U)$ is open. Then we must show that $U$ is open in the Hausdorff topology. Consider a point $p\in U$. If $\#(p)= k$ then $q_k^{-1}(p)$ has exactly one point which we also write as $p$. Also since $q_k ^{-1} (U)$ is open,  an open ball of $B(p,\delta)$ is contained in it. 

Let the minimum distance between points of $p$ be $\lambda$. We note that if  $z$ is at 
distance less than $\lambda / 2$ from $p$ then it has the same cardinality as $p$ and $d_k(p,z)=d_H (p,z)$. 
Therefore, if $r=min(\delta, \lambda / 2)$, the ball $B(p,r)$ is contained in $q_k^{-1} ( U)$. The ball in the Hausdorff metric of radius $r$ is equal to
$q_k(B(p,r))$ which is contained in $U$ since $r \leq \delta$, showing that $U$ contains an open neighborhood of $p$. 

Now we let $p$ have cardinality $< k$ so that 
$q_k^{-1}(p)$ has more than one point. Let $q_k^{-1}(p) = p_1,p_2,\ldots ,p_n$. Since $q_k^{-1}(U)$ is open, 
there exists $\ep > 0$ such that \\
(1) $q_k^{-1}(U)$ contains an $\ep$ neighborhood of each of the $p_i$ (we can choose this uniformly since $n$ is finite), and \\
(2) $2\ep$ is less than the minimum distance between points of $p$. 

Now suppose there exists $z$ in the $\ep$ neighborhood of $p$. 
The second condition ensures that a point in $z$ is $\ep$ close to a unique point of $p$ 
giving rise to
 a well-defined function $\Lambda:z \rightarrow p$. 
Hence, any weighted product of elements of $z$ so that the total weight is $k$ and
so that each $z$ point has a non zero weight, is $\ep$ close in the metric $d_k$ to the weighted product of elements  obtained 
by applying $\Lambda$ pointwise. 
Since every point in $p$ is $\ep$ close to at least one point of $z$, the element thus obtained maps to $p$. 
Thus, there is one $z_i \in q_k^{-1}(z) $ which is $\ep$ close to $p_i$, implying $q \in U$. Therefore $U$ contains an $\ep$ neighborhood of $p$. This completes the proof. 
\end{proof}

\subsection{Hypotheses for Existence of Equilibria}\label{hyp}

The next definition establishes the generality in which we shall work in this paper. 

\begin{defn} We say that a function 
 $f:\A_{-i}\ra \HH_c(\A_i)$ can be subset-lifted to $C_M(A_i)$ if there exists a continuous map $f^M :\A_{-i}\ra C_M(A_i)$ such that for every $a\in \A_{-i}$,
$$q^M(f^M(a))\subset f(a)\in \HH_c(\A_i).$$
\end{defn}

\begin{defn} A cost function  $r_i: \A_i \times \A_{-i}\ra \R$ is said to be {\bf configurable} if the response function
$R_i:\A_{-i} \rightarrow \HH_c(\A_i) $  can be subset-lifted to a map 
$R_i^M:\A_{-i} \rightarrow C_M(\A_i) $ to some configuration space $ C_M(\A_i)$. \\

A cost function  $r_i: \A_i \times \A_{-i}\ra \R$ is said to be {\bf weakly configurable} if for all $\ep > 0$ there exists
$R_{i,\ep}:\A_{-i} \rightarrow \HH_c(\A_i) $ such that 
\begin{enumerate} 
\item $R_{i,\ep}$ can be subset-lifted to a map 
$R_i^M:\A_{-i} \rightarrow C_M(\A_i) $ for some configuration space $ C_M(\A_i)$,
\item $d_H(R_i(a_{-i}) , R_{i,\ep}(a_{-i})) < \ep$ for all $a_{-i}\in A_{-i}$.
\end{enumerate}
\label{pol}
 \end{defn}

Equivalently, a cost function $r_i$ is configurable if\\
1) the set of local minima $R_i(a_{-i})$
of the function $r_i(a_{i},a_{-i})$ is finite for each $a_{-i}\in \A_{-i}$; and \\
2) Counted with multiplicity, the set of local minima $R_i(a_{-i})$ is  continuous on $ \A_{-i}$.\\
Also, a weakly configurable cost function is one whose set of local minima can be arbitrarily well-approximated by  
the set of local minima of a  configurable  function.

The notion of a weakly configurable cost function gives us greater flexibility than that of a configurable cost function.
It is possible that for a weakly configurable cost function, 
the set of local minima $R_i(a_{-i})$
of the function $r_i(a_{i},a_{-i})$ is infinite,  i.e. for the sequence $f^M$ of configurable approximants, $M \rightarrow \infty$.
The simplest example of $r_i: \A_i \times \A_{-i}\ra \R$ that is weakly configurable but not configurable, is a constant function 
or more generally a function that is independent of the first variable.

\medskip

There are two assumptions that will be relevant in this paper. \\
\noi \textbf{(A1)} For each $i$, the map $R_i$ is continuous in $a_{-i}$.\\
\noi \textbf{(A2)} For each $i$, the map $R_i$ is weakly configurable in $a_{-i}$.\\

The main aim of this paper is to prove the existence of Nash equilibria under the assumption \textbf{(A2)}.  
In Section \ref{ctreg} we shall give a counter-example to show that assumption \textbf{(A1)} is not sufficient
to guarantee  the existence of Nash equilibria.


\medskip

\noindent {\bf Motivation and Examples:}  The motivation for Definition \ref{pol} comes from Complex Algebraic Geometry.
Bezout's Theorem, which counts (with multiplicity) the number of intersections of projective hypersurfaces,
 has already been used in the context of Nash equilibria \cite{mcl}. The notion of complex Nash equilibria
has also been investigated by the authors of  \cite{mckl}. 

At the core of the Theorem that Nash equilibria exist 
under assumption \textbf{(A2)} (Theorem \ref{existNash}) lies an argument that derives from the algebraic topology principles that prove Bezout's
Theorem, viz. Poincar\'e duality and intersection theory of algebraically defined cycles.

The other main ingredient of the proof, viz. existence of multiselections, is also motivated by  Complex Algebraic Geometry.
Natural examples of  maps $f: X \rightarrow Y$ admitting multisections (counted with multiplicity)  occur as follows in the context of 
 Complex Algebraic Geometry:\\
Let $f: X \rightarrow Y$ be a surjective morphism of complex projective manifolds. Let $d = dim(X) - dim(Y)$.
One can cut $X$ by hyperplanes. Cutting $X$ with $d$ generic hyperplanes  furnishes
 a multisection (counted with multiplicity) of $f$. Thus, $F: Y \rightarrow {\mathcal{H}}_c (X)$ defined by $F(y) = f^{-1}(y)$ always admits
a subset-lift $F^M: Y \rightarrow C_M (X)$ for some $M$. We were led  to  Definition \ref{pol} by such  examples.


\section{Homological Selection}\label{hom}

In this section we will prove a result about the homology of  $gr(R_i)$  under the assumption (\textbf{A2}). The result will be used to prove the existence of Nash equilibria in the next section. We shall use in the following Homology with {\it integer} coefficients. Thus, $H_\ast(X)$ (resp. $H_\ast (X,A)$) will be used to denote $H_\ast(X; \mathbb{Z})$ (resp. $H_\ast (X,A; \mathbb{Z})$).

\smallskip

For a topological space $X$ with a basepoint $*$, we can use the basepoint to define a canonical inclusion of topological spaces $C_r(X)\hookrightarrow C_{r+1}(X)$ by the formula
$$x_1x_2\ldots x_r \mapsto x_1x_2\ldots x_r *$$
In this expression we are using the multiplicative notation defined earlier. The space $C_r(X)$ is also called the ($r$-fold) symmetric product and 
is often denoted by $SP_r(X)$ (as in \cite{hatcher}). 

It is interesting to form the union (direct limit) of the spaces $C_r(X)$, $C_\infty(X)$. This is called the infinite symmetric product and 
is denoted $SP_\infty(X)$. The space obtained in the union is the free commutative monoid on the space $X$ with the relation that the basepoint $*$ is the identity. 

In terms of this paper our interest in the space $C_\infty(X)$ is in the fact that the homotopy groups of $C_\infty(X)$ are the homology groups of $X$.
This is called the Dold-Thom theorem (\cite{hatcher} Section 4.K). Moreover the Hurewicz homomorphism from $\pi_*(X)\rightarrow H_*(X)$ is 
induced by the inclusion  
$$X= C_1(X)\hookrightarrow C_\infty(X)$$
on homotopy groups. 

The advantage of this space is that in order to show that a certain cycle is zero in homology, it suffices to establish a null-homotopy of the corresponding map into the configuration space. This is exactly what works in the proof of the following lemma. We are grateful to A. Dranishnikov for telling us the proof-idea.    

For the lemma below, we define the graph of a function $f:\DD^m \rightarrow C_r(\DD^n)$ to be the subset of $\DD^m \times \DD^n$ given by
$$gr(f)= \{ (x,y): x \in \DD^m, y \in q_r \circ f(x) \subset \DD^n\}.$$

\begin{lemma}\label{Dold}
Suppose that $f: \DD^m \rightarrow C_r(\DD^n)$ has the same value $(0,\ldots,0)\in C_r(\DD^n)$ on the boundary $\partial \DD^m =S^{m-1}$. Consider the map $i: S^{m-1} \rightarrow gr(f)$ with $i(x)= (x,0)$. Then $i_*$ induces the zero map on rational homology, or equivalently,
$$ i_* \otimes \mathbb{Q} : H_{m-1}(S^{m-1}) \otimes \mathbb{Q} \rightarrow H_{m-1}(gr(f)) \otimes \mathbb{Q}$$
is zero.
\end{lemma}

\begin{proof}
We will apply the Dold-Thom theorem. The reduced homology $\tilde{H}_*(S^{m-1})$ is concentrated in degree $m-1$, in which degree it is $\mathbb{Z}$. The generator is in the image of the Hurewicz homomorphism $\pi_{m-1}S^{m-1} \stackrel{\cong}{\rightarrow} H_{m-1}(S^{m-1})$, and is the image of the identity
 element of $\pi_{m-1}S^{m-1}$
under the Hurewicz map.

To the function $f:\DD^m\rightarrow C_r(\DD^n)$ we associate a map $\beta(f):\DD^m \rightarrow C_r(gr(f))$ (the $r$-point configuration space of the graph of $f$)
as follows. Writing the elements of the configuration space using the multiplicative notation, we define
$$ \beta(f)(x) = (x,y_1)^{i_1} \ldots (x,y_k)^{i_k} \in C_r(gr(f))$$
if $f(x)= y_1^{i_1}\ldots y_k^{i_k}$. This induces a continuous embedding of the disk $\DD^m$ in the $r$-point configuration space of the graph of $f$. 
The induced map on the boundary $S^{m-1}$ is the diagonal embedding $e_r(i): x\mapsto (x,0)^r$.

Let $e_r$ be the composition
$$S^{m-1}\stackrel{\Delta}{\rightarrow} \times_r S^{m-1} \rightarrow C_r(S^{m-1}),$$ where $\Delta$ 
denotes the diagonal embedding of $S^{m-1}$ into the $r$-fold product $\times_r S^{m-1}$ of $S^{m-1}$ with itself.
Now $e_r(i)$
 factors as the composite $ C_1(S^{m-1}) ( \simeq S^{m-1})\stackrel{e_r}{\rightarrow} C_r(S^{m-1}) \stackrel{C_r(i)}{\rightarrow} C_r(gr(f))$,
where $C_r(i)$ denotes the map on $r-$fold configuration spaces induced by $i$.
 We also have the inclusion $i_r : S^{m-1} \simeq C_1(S^{m-1}) \rightarrow C_r(S^{m-1})$ defined as $i_r(x) = x.*^{r-1}$ ($*$ a fixed base-point of $S^{m-1}$). (Note that this is defined also for $r=\infty$ as the inclusion $C_1(S^{m-1}) \rightarrow C_\infty(S^{m-1})$)

We complete the proof modulo Sub-lemma \ref{subl}
which states  that $e_r\simeq r. i_r$, or equivalently that the diagram
$$\xymatrix{ S^{m-1} \ar[r]^{\simeq}\ar[d]^{\phi_r} & C_1(S^{m-1})  \ar[d]_{e_r} \\ 
              S^{m-1}\ar[r]^{i_r}                       & C_r(S^{m-1})    }$$
is homotopy commutative, where degree of $\phi_r=r$. (Note that in the homotopy group $\pi_{m-1}X$, $f\circ \phi_r \simeq rf$ by definition.)

 We have the commutative diagram
$$\xymatrix{ S^{m-1} \ar[r]^{\simeq} & C_1(S^{m-1}) \ar[r]^{i} \ar[d]_{i_r}  & C_1(gr(f)) \ar[r]^{\simeq} \ar[d] & gr(f) \\ 
                                     & C_r(S^{m-1})\ar[r]^{C_r(i)} \ar[d]                        & C_r(gr(f)) \ar[d]               &  \\                                                        & C_{\infty}(S^{m-1}) \ar[r]^{C_\infty(i)}                       & C_{\infty}(gr(f))               & }$$

We know that $\pi_*C_\infty(X) \cong H_*(X; \mathbb{Z})$ by the Dold-Thom Theorem. Also, $i_* : H_{m-1}(S^{m-1}) \rightarrow H_{m-1}(gr(f))$ is the same 
as the map $\pi_{m-1}C_\infty (i): \pi_{m-1} C_\infty(S^{m-1}) \rightarrow \pi_{m-1} C_\infty(gr(f))$ induced between the $(m-1)$th homotopy groups.

Next note that $\beta(f)$ gives a null-homotopy of $e_r(i)$ and hence of $C_r(i) \circ e_r\simeq  C_r(i) \circ i_r \circ \phi_r$. 
$$\xymatrix{  C_1(S^{m-1})  \ar[dd]_{e_r\simeq ri_r} \ar[rrrdd]^{e_r(i)}   \\ 
\\
              C_r(S^{m-1})\ar[rrr]^{C_r(i)}                      &&  & C_r(gr(f))               }$$
\

From the commutative diagrams above we get that 
\begin{eqnarray*}
  & C_r(i)\circ i_r \circ \phi_r & \simeq 0 \\
\Rightarrow & C_\infty(i)\circ i_\infty \circ \phi_r & \simeq 0 \\
\Rightarrow & r.i_* & =  0 \\
\Rightarrow & i_* \otimes \mathbb{Q} & =0,
  \end{eqnarray*}
where the second implication follows from the commutativity of the (big) three-line diagram above and the third implication follows from the Dold-Thom theorem.
This completes the proof.
\end{proof}

\begin{sublemma} $e_r\simeq r. i_r$. \label{subl} \end{sublemma}

\begin{proof}
 Both $e_r$ and $r. i_r$ are elements of $\pi_{m-1}C_r(S^{m-1})$. A
  CW complex structure on $X$ induces a CW complex structure on $C_r(X)$ and $C_\infty(X)$ 
(cf.  \cite{hatcher}) 
which makes the latter $k$-connected if $X$ is. Therefore, the
 Hurewicz theorem implies that $\pi_{m-1}C_r(S^{m-1}) \cong H_{m-1}(C_r(S^{m-1}))$. Hence it suffices to show that ${e_r}_*=r{i_r}_*$ on homology. 

Note that $H_{m-1} (C_r(S^{m-1})) \cong \mathbb{Z}$ and ${i_r}_*$ induces an isomorphism $H_{m-1}(S^{m-1})\rightarrow H_{m-1}(C_r(S^{m-1}))$. We may write $e_r$ as a composite 
$$S^{m-1}\stackrel{\Delta}{\rightarrow} \times_r S^{m-1} \rightarrow C_r(S^{m-1})$$
where the first map is the diagonal map, and the second is the canonical projection. On $H_{m-1}$, $\Delta_*$ takes $1\in H_{m-1}(S^{m-1})$ to $(1,\ldots,1)\in H_{m-1}(\times_rS^{m-1})=\mathbb{Z}^r$. 

The (m-1)-skeleton of $\times_rS^{m-1}$ is $\vee_r S^{m-1}$ (under the standard CW complex structure on the product induced from the CW complex structure on $S^{m-1}$ with one m-1-cell). The projection map induces $i_r$ on each factor. Therefore the projection map on $H_{m-1}$ is the sum $\mathbb{Z}^r \rightarrow \mathbb{Z}$ given by $(a_1,\ldots,a_r)\mapsto a_1+\ldots+a_r$. Therefore 
$${e_r}_*(1)=r$$
as required. 
\end{proof}

\begin{rmk}{\label{min}}
In the case $n=1$ of the above lemma, we get an easier proof which works for any map $\DD^m \rightarrow Sub_r(I)$. In this case the map which sends a subset to its minimum is a continuous function from $Sub_r(I)$ to $I$. Therefore we have a disk ($x\mapsto (x,min(f(x))$) in $gr(f)$ whose boundary is $S^{m-1}= \partial gr(f)$. This shows that $i_* =0 $ and hence, $i_* \otimes \mathbb{Q} =0$.   
\end{rmk}

\begin{rmk}{\label{linear}}
In the case $m=1$ also the above conclusion is correct for any map $I\rightarrow Sub_r(\DD^n)$. The graph $gr(f)$ is path connected so one can choose a path from the point over -1 to the point over 1. Removing loops one obtains a 1-cell with boundary $\partial gr(f)$. 
\end{rmk}

The conclusion of the lemma above plays a crucial role in proving the existence of Nash equilibria in the next section. Observe that the map $i_*\otimes \mathbb{Q}=0$ implies that the generator of $H_{m-1}(S^{m-1})$ maps to a torsion class in $H_*(gr(f))$ (a class whose multiple is 0), and therefore the map $i_*$ cannot be injective.
From the long exact sequence of homology groups we conclude that $H_m(gr(f),\partial gr(f))\rightarrow H_{m-1}(\partial gr(f))$ is non zero. This will be an
ingredient in Theorem \ref{multisel}.    

In order to arrive at this conclusion, it also suffices to have an embedded disk in $gr(f)$ lying over the disk $\DD^m$ (that is, a section of the projection map $gr(f)\rightarrow \DD^m$ which sends the element $(x,y);y\in f(x)$ to $x$) . This is classically termed as a ``Selection" (cf. \cite{michael1, michael2, michael3}).
 The conditions we are working with in this paper are not enough to obtain a true "selection".
However, Theorem \ref{multisel} below proves the existence  of a special homology class which serves as a replacement and which
we call a ``Homological Selection".

\begin{defn}\label{homseldef}
Suppose that $f: \DD^m \rightarrow \HH_c(\DD^n)$ is a continuous function from $ \DD^m$ into the space of compact subsets of
$ \DD^n$. 
Let $gr(f) = \lbrace (x,y) \in \DD^m \times \DD^n\, |\, y \in \, f(x)\rbrace$
and let $gr(\partial f) = \lbrace (x,y) \in \partial \DD^m \times \DD^n\, |\, y \in \, f|_{ \partial \DD^m}(x)\rbrace$, where 
$f|_{ \partial \DD^m}$ denotes the restriction of $f$ to ${ \partial \DD^m}$. 
 A non-zero  
$m-$dimensional chain $c_m$ supported in $gr(f)$ is said to be a {\bf homological selection} of $f$ if its boundary
$\partial^m c_m$ is supported in $gr(\partial f)$ and the projection $H_m(gr(f),gr(\partial f)) \rightarrow H_m(\DD^m,\partial \DD^m)$ maps $c_m$ to a non zero class. In such a situation we say that $f$ admits a  homological selection.
\end{defn}

\begin{rmk} In the definition above, a homological selection actually refers to a choice of a finite number of image points 
counted
with (positive integral) multiplicity,
for every point in the domain of a multifunction.\\
We reiterate that we are working with integral homology classes unless otherwise stated.  \end{rmk}

Theorem \ref{multisel} below is now a consequence of Lemma \ref{Dold}.

\begin{theorem} {\bf Homological Selection Theorem:}
Suppose that $f: \DD^m \rightarrow C_r(\DD^n)$ is a continuous function that has the same value $(0,\ldots,0)$ on the boundary $\partial \DD^m =S^{m-1}$.
Then $f$ admits a  homological selection. \label{multisel}
\end{theorem}

\begin{proof}

We know from Lemma \ref{Dold} that the map 
$$H_{m-1}(S^{m-1})\otimes \mathbb{Q} \cong H_{m-1}( gr(\partial f)) \otimes \mathbb{Q} \rightarrow H_{m-1}(gr(f))\otimes \mathbb{Q}$$ induced by the
inclusion of $gr(\partial f)$ into $gr(f)$
is zero. We consider the commutative square 
   $$\xymatrix{ H_m(gr(f),\partial gr(f))  \ar[r]^{\partial} \ar[d]_{{\Pi}_*}  &  H_{m -1}( gr(\partial f)) \ar[d]^{\cong} \\ 
             H_m(\DD^m,\partial \DD^m) \ar[r]^{\cong}                 &  H_{m -1}(\partial \DD^m)              },$$
where $\Pi$ is the natural projection from $gr(f)$ to $\DD^m$.
It suffices to show that the boundary map $\partial$ from $H_m(gr(f),\partial gr(f))$ to $ H_{m -1}( gr(\partial f))$ is non zero. The latter group is the top homology of the sphere, hence $\cong \mathbb{Z}$. The next term in the long exact sequence is $H_{m -1}(gr(f))$. 

Hence it is enough to show that
 the map $ H_{m -1}(\partial gr(f)) \rightarrow H_{m -1}(gr(f))$ has non zero kernel. Since 
 the  group
$ H_{m -1}(\partial gr(f))$ is $\mathbb{Z}$, it suffices to show
 that this map is the zero map after tensoring with $\mathbb{Q}$. This is precisely what Lemma \ref{Dold} ensures and we are done. 
\end{proof}
\begin{rmk} The hypothesis that $f: \DD^m \rightarrow C_r(\DD^n)$  has the same value $(0,\ldots,0)$ on the boundary $\partial \DD^m =S^{m-1}$
can be dropped. Any $f: \DD^m \rightarrow C_r(\DD^n)$ can be extended by a linear homotopy to such a function $f_1$
by enlarging $\DD^m$ slightly by adding an annulus $A_m$. We note that $gr(A_m)=\lbrace (x,y) \in A_m \times \DD^n\, |\, y \in \, f|_{ A_m}(x)\rbrace$ deformation retracts to the outer boundary $S^{m-1}$. Therefore,
 $$H_m(gr(f_1), gr(\partial f_1)) \cong H_m(gr(f_1), gr(A_m)) \cong H_m(gr(f),gr(\partial f))$$
  the second isomorphism coming from excision. Since $f_1$ admits  a  homological selection by Theorem \ref{multisel} the same class under the above isomorphisms gives  a  homological selection for $f$.
\end{rmk}

Next we apply the Homological Selection Theorem in the context
 of Nash equilibria. We use the following notation: for any set $A$ in a metric space $(X,d)$ and any $\ep > 0$, the $\ep$-neighborhood of $A$ in $X$ is written as $A^\ep$. Recall that for a cost function $r_i:\A=\A_i\times \A_{-i}\ra \R$ we have a best-response map
$$R_i:\A_{-i}\ra \HH_c(\A_i)$$ 
which we have assumed is a continuous function. 

In the next Proposition we show that the assumption that 
$r_i$ is weakly configurable allows us to apply Theorem \ref{multisel}. The hypothesis that
 $r_i$ is weakly configurable  ensures (given $\ep > 0$) the existence of  some $M$, and
 an $R_i^M:\A_{-i}\ra C_M(\A_i)$, such
 that $d_\HH(R_i(a),q_M \circ R_i^M(a))<\ep$ $\forall a \in \A_{-i}$ (recall that $q_M$ is the map from $C_M(\A_i)$ to $\HH_c(\A_i)$ that forgets
multiplicities). 
The latter condition implies that $gr(R_i^M)\subset gr(R_i)^\ep$. The Homological Selection Theorem \ref{multisel}
guarantees a Homological Selection  for $gr(R_i^M)$. This implies a Homological
 Selection for $gr(R_i)^\ep$ and is explicated in the Proposition below.  

For any $\ep > 0$ and for all $i$, $\mbox{gr}(R_i)^\ep$ is intrinsically a manifold with boundary and an open subset of $\SM$. Let $\Pi_i$ denote the (continuous) projection $\Pi_i : \mbox{gr}(R_i)^\ep \ra \SM_{-i}$. Let $\mbox{dim}(\SM_{-i}) =\sum_{j=1,j\neq i}^N n_j \equiv d_i$. Also , let $d_0 = \sum_{j=1}^N n_j $ so that $d_i = (d_0 -n_i)$.
As usual,  we denote the  boundary  of a manifold $M$ by $\partial M$.   

\begin{prop} \label{homocohomomap} Suppose that a
 cost function  $r_i: \A_i \times \A_{-i}\ra \R$ is  weakly configurable. Let $R_i$ be the response function of $r_i$. Then
there exists $\ep_0 > 0$ such that for all $0 < \ep < \ep_0$ and for all $i$, the induced homomorphism of $\Pi_i$ in the relative (co)homology \[\Pi_i^* : H^{d_i}\left(\SM_{-i},\partial \SM_{-i}\right) \lra H^{d_i}\left(\mbox{gr}(R_i)^\ep,\partial \mbox{gr}(R_i)^\ep\right)\] or equivalently \[\Pi_{i*} : H_{d_i}\left(\mbox{gr}(R_i)^\ep,\partial \mbox{gr}(R_i)^\ep\right) \lra H_{d_i}\left(\SM_{-i},\partial \SM_{-i}\right)\] is nonzero.
\end{prop}

\begin{proof}
Firstly, suppose that there is a continuous function $f$ from $\SM_{-i}$ to $C_M(\A_i)$ so that 
$$q_M(f(a))\subset R_i(a)^\ep \ \forall a\in \SM_{-i}.$$ (As we shall see, the weakly configurable condition guarantees the existence of such an $f$.)
Applying Theorem \ref{multisel} to $f$ we have that the map
$H_{d_i}(gr(f),\partial gr(f)) \rightarrow H_{d_i}(\SM_{-i},\partial \SM_{-i})$ is non zero. 

Note that the above condition on $f$ 
implies that $gr(f)\subset gr(R_i)^\ep$. Therefore, at the level of homology we have a composition of maps 
$$H_{d_i}(gr(f), \partial gr(f)) \rightarrow  H_{d_i}\left(\mbox{gr}(R_i)^\ep,\partial \mbox{gr}(R_i)^\ep\right) \lra H_{d_i}(\SM_{-i},\partial \SM_{-i})$$
Now from Theorem \ref{multisel} we have that the composite is non-zero, which implies that the second map is non-zero. This is the second part of the
statement of the
Proposition. The first part is an equivalent formulation in terms of cohomology.

It thus remains only to establish the existence of a suitable $f$. This is guaranteed by the weakly configurable condition. Since $r_i$ is weakly configurable, for every $\ep > 0$
 there exists $R_{i,\ep}$ such that\\
a) $d_H(R_{i,\ep}(a),R_i(a))<\ep$ for all $a\in \A_{-i}$, and \\
b) $R_{i,\ep}$ can be subset-lifted to a map $R_i^N$ into the space $C_N(\DD^n)$. \\
 Choose $f=R_i^N$. Then condition a) implies that $q_M(f(a))\subset R_i(a)^\ep \ \forall a\in \SM_{-i}$
as required.  This completes the proof.   
 \end{proof}

\section{Existence of Nash Equilibria}
In this section, we prove the existence of Nash equilibria under the assumption (\textbf{A2}).

For any $k \in \mathbb{N}$ we denote $[k] = \{1,2,\ldots,k\}$. We begin with the following useful lemma.

\begin{lemma} \label{hausintersect}
Given closed subset sequences $\{A_{ji}\}_{i \in \mathbb{N}},j \in [k]$ of a compact subset $S$ of $X$ such that $\bigcap_{j}A_{ji} \neq \emptyset,\forall i$ and $d_{H}(A_{ji},A_j) \stackrel{i \ra \infty}{\lra} 0$ for given non-empty compact subsets $A_j,j \in [k]$ of $S$, then $\bigcap_j A_j \neq \emptyset$.
\end{lemma}

\begin{proof} Let $x_i \in \bigcap_{j}A_{ji}$. Passing to a subsequence if necessary, $x_i \ra x_\infty \in S$.
It suffices to show that $x_\infty \in A_j$ for all $j= 1,2,\ldots,k$.

Suppose not, i.e. there exists $j$ such that $x_\infty \notin A_j$. Let $d(x_\infty, A_j) = \ep > 0$.
Choose $i_0$ such that $d_{H}(A_{ji},A_j) \leq \frac{\ep}{2}$ for all $i \geq i_0$. Then $d(x_\infty, A_{ji}) \geq \frac{\ep}{2}$
for all $i \geq i_0$. In particular, $d(x_\infty, x_i) \geq \frac{\ep}{2}$
for all $i \geq i_0$, contradicting the fact that $x_i \ra x_\infty \in S$. 
 \end{proof}

Before we state our main theorem we  recall the notion of
{\bf external cup product, or cross product} of relative cohomology classes (\cite{hatcher} p. 220).
Since $\A = \A_i \times  \A_{-i}$, we have an 
 isomorphism $H^{n_i}(\A_i, \partial \A_i) \otimes H^{d_i}(\A_{-i}, \partial \A_{-i})\ra H^{d_0}(\A, \partial \A)$  (\cite{hatcher} Theorem 3.20).
This isomorphism is implemented  by the cross product $a \times b = p_1^\ast (a)\cup p_2^\ast (b)$, where
$p_1: (\A_i \times \A_{-i}, \partial \A_i\times \A_{-i}) \ra (\A_i, \partial \A_i)$ and 
     $p_2: (\A_{-i} \times \A_i, \partial \A_{-i} \times \A_i)\ra (\A_{-i}, \partial \A_{-i})$ are projections of pairs.
Also, since $H^{n_i}(\A_i, \partial \A_i) = H^{d_i}(\A_{-i}, \partial \A_{-i})= H^n(\A, \partial \A) = \mathbb{Z}$,  (relative) Poincar\'e
Duality holds in this set up.

Using Definition \ref{optrespgr} the existence of Nash equilibria can be stated as follows.

\begin{theorem} \label{existNash} Consider an $N-$person non-cooperative game with mixed strategy spaces satisfying assumption 
\textbf{(A2)}. Let $r_1, \cdots, r_N$ be the cost functions and  $R_1, \cdots, R_N$ be the response functions. Then,
$\bigcap_{i=1}^N \mbox{gr}(R_i) \neq \emptyset$. Equivalently, the game has at least one Nash equilibrium.
\end{theorem}

\begin{proof}  We continue with the notation of Proposition \ref{homocohomomap}, which gives the following:
\[\Pi_{i*} : H_{d_i}\left(\mbox{gr}(R_i)^\ep,\partial \mbox{gr}(R_i)^\ep\right) \lra H_{d_i}\left(\SM_{-i},\partial \SM_{-i}\right)\] is nonzero.
For each $i$, choose a relative $d_i-$cycle $z_i \in H_{d_i}\left(\mbox{gr}(R_i)^\ep,\partial \mbox{gr}(R_i)^\ep\right)
(\subset  H_{d_i}(\A_{-i}, \partial \A_{-i}))$ such that
$w_i=\Pi_{i*} (z_i) \neq 0$.

Let $w_i^\ast (\neq 0) \in H^{n_i}(\A_i, \partial \A_i)  $ be its  (relative) Poincar\'e Dual. Then $\bigcup_{i=1 \cdots N}  
w_i^\ast \in H^{d_0}(\A, \partial \A)$ is a non-zero cohomology class. Hence, by (relative) Poincar\'e
Duality again the intersection of the supports of the chains $z_i$ is non-empty, i.e. $\bigcap_{i=1 \cdots N}  \mbox{supp} (z_i)
\neq \emptyset$. One way to see this  is to take simplicial approximations of the chains $z_i$ homologous to $z_i$
contained in $\mbox{gr}(R_i)^\ep$ and ensure that their supports are in general position. Then $\bigcap_{i=1 \cdots N}  \mbox{supp} (z_i)$
is the support of a zero-cycle (Poincar\'e) dual to 
$\bigcup_{i=1 \cdots N}  
w_i^\ast$. Hence the intersection of the closures $Cl(\mbox{gr}(R_i)^\ep)$ is non-empty, i.e. $\bigcap_{i=1 \cdots N}  Cl(\mbox{gr}(R_i)^\ep)
\neq \emptyset$. Since this is true for all $\ep > 0$, it follows from Lemma \ref{hausintersect} that 
$\bigcap_{i=1 \cdots N}  \mbox{gr}(R_i)
\neq \emptyset$. Equivalently, the game has at least one Nash equilibrium.
\end{proof}

\section{A Counterexample}\label{ctreg}

In this section we give an example that  satisfies
 (\textbf{A1}) but not (\textbf{A2}). In the example,
the conclusion in Proposition \ref{homocohomomap} is no longer true.
This stresses the necessity of the assumption
(\textbf{A2}). Later in the section we use the example to construct continuous cost functions for which Nash equilibria don't exist. 

Recall that the assumption ({\bf A2}) states that the subset valued function can be lifted to the configuration space in a neighborhood. For the counterexample, we use finite subset values (that is, with values in $Sub_k$). We attempt to construct a function  $\DD^n\ra Sub_k(\DD^m)$ which does not satisfy ({\bf A2}).  In other words we wish to construct $f$ for which there are no lifts $\tilde{f}$ in the diagram
$$\xymatrix{  & & & C_N(\DD^m)\ar[d] \\ 
  \DD^n \ar[rr]_{f} \ar[rrru]^{\tilde{f}} & & Sub_k(\DD^m)\ar[r] & Sub_N(\DD^m)} $$
Recall that a point in the configuration space $C_N(\DD^m)$ is $x_1^{i_1}\cdots x_r^{i_r}$ with $x_j\in \DD^m$ such that $\sum i_j = N$. For a point $x\in \DD^n$ so that $f(x)=\{x_1,\cdots,x_k\}$ with $x_j\in \DD^m$, a lift $\tilde{f}(x)$ of $f(x)$ is of the form $x_1^{i_1}\cdots x_k^{i_k}$ such that $\sum i_j = N$. We think of the integers $i_j$ as weights assigned to the various points of the subset-valued function $f$, and we think of the lift $\tilde{f}$ to be a coherent assignment of weights.

Proposition \ref{subconf} indicates that the topology on $Sub_k(\DD^m)$ is the quotient topology inherited from $C_n(\DD^m)$ by forgetting weights/multiplicities. Observe from this that there are no trivial counterexamples as the constant maps will always lift.

In the following we need some more notation. If $j: A\to B$ is a continuous function then we get a map $Sub_k(j):Sub_k(A)\to Sub_k(B)$ given by 
$$Sub_k(j)(\{x_1,\cdots , x_k\})= \{j(x_1),\cdots, j(x_k)\}$$ 
In the special case that $j$ is an inclusion the map $Sub_k(j)$ is also an inclusion. For this reason we shall write for $A\subset B$, $Sub_k(A)\subset Sub_k(B)$. 

We shall also use the notion of linear homotopy below. Suppose that a space $A$ deformation retracts to a point $a\in A$. Write $H:A\times I \to A$ as $H(x,t)= H_t(x)$. This induces a deformation retraction $Sub_k(H)$ of $Sub_k(A)$ to $\{a\}\in Sub_k(A)$ by the formula
$$Sub_k(H)(\{x_1,\cdots,x_r\},t) = \{H(x_1,t),\cdots,H(x_r,t)\}$$
In the special case $A=\DD^2$, we call $H(x,t)=x(1-t)$ the linear homotopy of $\DD^2$ onto $(0,0)$. The corresponding deformation retraction $Sub_k(H)$ of $Sub_k(\DD^2)$ is also called the linear homotopy.

The example below will involve a map $\DD^2 \rightarrow Sub_3(\DD^2)$ (that is, $m=n=2$). Note that this is the `smallest 
complexity' case where such an example can exist: by Remark \ref{min} the conclusion in Proposition \ref{homocohomomap} is satisfied if $n=1$. By Remark \ref{linear} we have the case $m=1$, and since $Sub_2(X) = C_2(X)$ we obtain the result for the case $\DD^2 \rightarrow Sub_2(\DD^2)$. We also use the counterexample to demonstrate that Nash equilibria might not exist  or equivalently that Theorem
\ref{existNash} does not hold if we assume only (\textbf{A1}). However, before we get into the counterexample, a remark is in order.

\begin{rmk} {\rm Let $\PP(\DD^j)$ denote the set of probability measures on $\DD^j$
equipped with the weak topology.
Also let $\PP_n(\DD^j)$ be the subspace of
probability measures supported on (at most) $n$ points.
 Any continuous function $f:\DD^l \ra Sub_n(\DD^j)$ can be lifted to $\PP_n(\DD^j)$ by Michael's Selection Theorem \cite{michael3}.
Equivalently any continuous function $f:\DD^l \ra Sub_n(\DD^j)$ can be lifted to a continuous 
 {\bf  weighted  real valued multifunction} with
total weight constant  at each point, where all weights  are non-negative.
The  following counterexample  shows that we cannot upgrade this 
to a continuous 
   weighted  {\bf rational-valued} (or equivalently integer-valued, after normalization) multifunction with
constant total weight at each point.}
\end{rmk}

\subsection{Counterexample to Proposition 
 \ref{homocohomomap}}\label{ctreg1}
We write down this counter-example in two steps. In order to get a counter-example to Proposition \ref{homocohomomap} we must have a map into the subset space to which a coherent assignment of weights is not possible. The first step constructs a map into $Sub_3(\DD^1)$ for which certain elements are forced to have weight zero. More precisely, the graph of a map into $Sub_3(\DD^1)$ consists of $3$ strands, and one of the strands is forced to have weight $0$.

\subsubsection{Step 1: Multifunctions from $S^1$ to $\DD^1$:}
 Consider the map $g:S^1 \ra Sub_3(\DD^1)$ given by $g(e^{2\pi i \theta}) = \{0, 1, \theta \}$, where $\theta \in [0,1]$ and $\DD^1$ is
identified with $[0,1]$. Any lift $\til{g}:S^1 \ra C_N(\DD^1)$ for some $N\geq 3$ assigns non-negative integer weights to each point in $g(x)$ for all $x \in S^1$. Choose a regular point $a \in S^1$ (for $g$), i.e. a point such that $\# g(a) = 3$. Let $g(a) = \{ a_1, a_2, a_3\}$ (see figure below). Let $w_i$ be the weight assigned to $a_i$ so that $(w_1+w_2+w_3)=N$.

By continuity, the weights on the strands containing each of the $a_i$'s is constant over all regular points. Let $b=e^{2\pi i}$
be the only non-regular point
and let $g(b) = \{ b_1, b_2\}$. Let $u_i$ be the weight assigned to $b_i$
so that $(u_1+u_2)=N$. Then $u_1=w_1 + w_2$ and $u_2=w_2+w_3$. It follows that $w_2=0$, i.e. the weight assigned to $a_2$ is zero.

\begin{center}
\includegraphics[height=70mm]{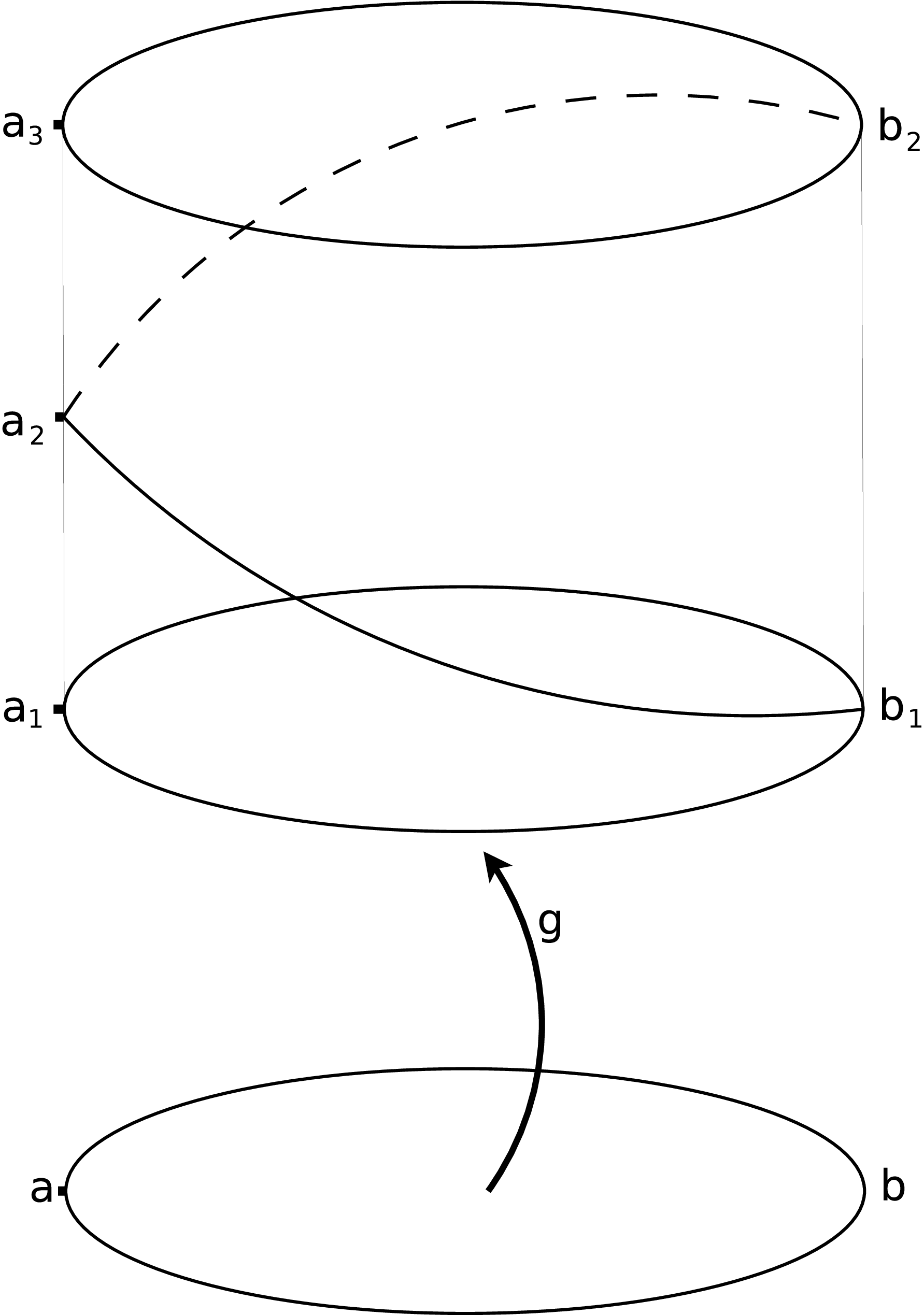}
\end{center}

\subsubsection{Scheme for Step 2:}\label{scheme}
We use this idea to arrange a counter-example for Proposition \ref{homocohomomap}. Note that the conclusion of Proposition \ref{homocohomomap} can be arrived at if the lift $\tilde{f}$ satisfies $q_k(\tilde{f}(x))\subset f(x)$ (that is, a subset of $f(x)$ can be assigned coherent weights). Therefore the above example will not suffice as the top and the bottom strands can be assigned arbitrary weights. 

The main idea behind the counterexample below is to {\it wedge three copies of the above picture together appropriately}. We want a strand arrangement over the disk so that the top and the bottom strands along one part of the disk gets associated to the middle strand along a different part of the disk. This will force that no subset can be assigned coherent weights.   

More precisely, wedge  three copies $S^1_1, S^1_2, S^1_3$ of the circle $S^1$ at the point $(0,0)$ inside the disk and fill up the disks $\DD^2_1,\DD^2_2,\DD^2_3$ inside as three petals. Choose an arbitrary
 circle $C$ in $\DD^2$ (to be identified with the {\it range/codomain} of the multifunction)
 and distinct 
points $a_1,a_2,a_3$ on  $C$. Along $\partial \DD^2_i$ (in the domain $ \DD^2_i$)
define the function $f_i$ as  above so that $f_i((0,0))=\{a_1,a_2, a_3\}$ is a regular point and $a_i$ corresponds to the middle strand. This can be done by choosing the path in the above example to be the path along $C$ going between the other two points passing through $a_i$ and parametrizing the angle $\theta$ appropriately. See diagram below, where $f_3$ is described.

\begin{center}
\includegraphics[height=90mm]{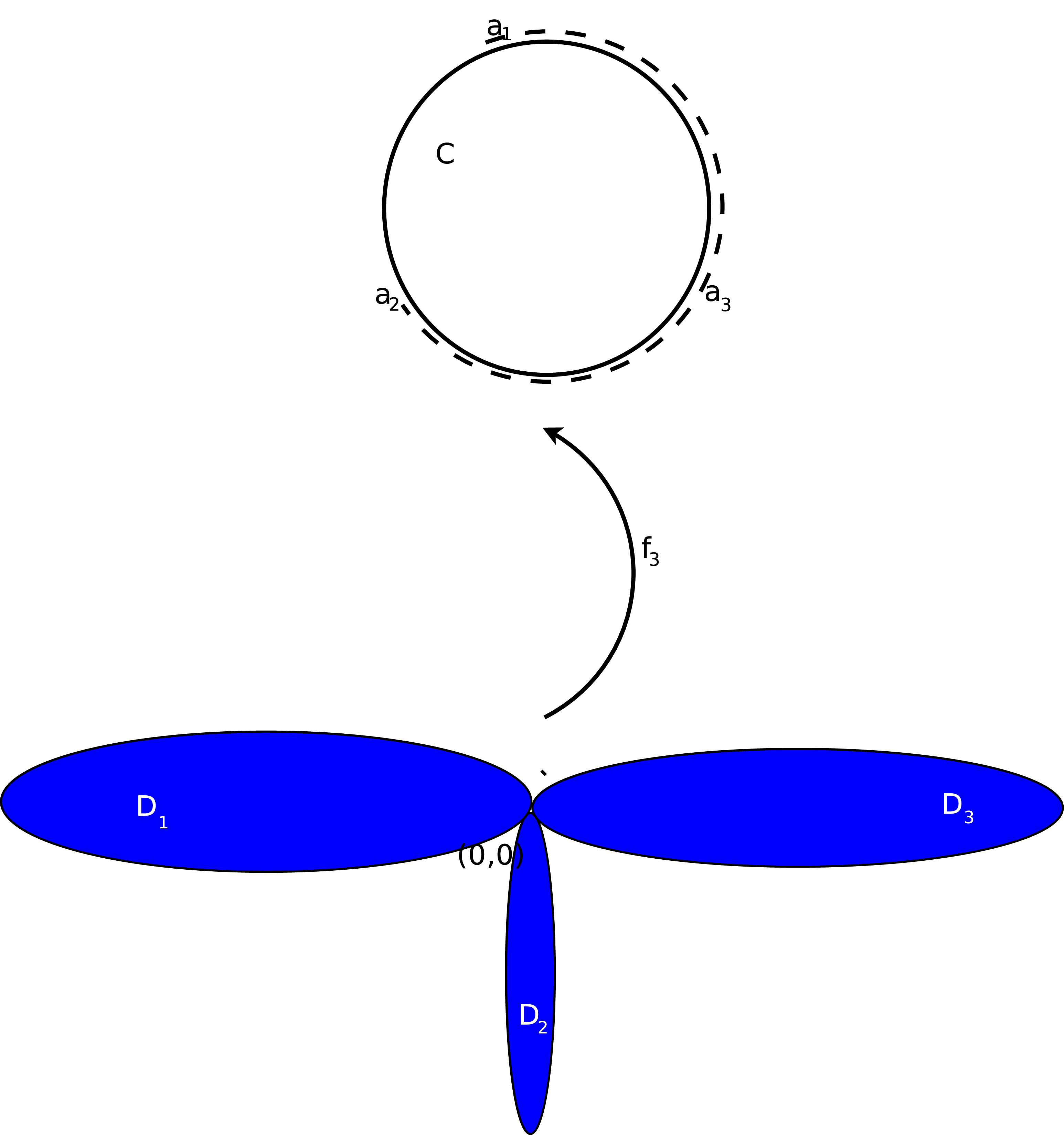}
\end{center}

Extend the function $f_i$ over all of $\DD^2_i$ by 
a "linear homotopy" (see Section \ref{linhom} below for details).
 Let $W= \DD^2_1\cup \DD^2_2 \cup \DD^2_3$.  Define the multi-function $g$ on $W$ by $g|_{\DD^2_i}=f_i$. Finally extend the multifunction $g$ to the whole disk
 $D$.  Since $a_i$ is forced to lie in the middle strand on $g(S^1_i)$, each weight $w_i$ is forced to be zero. In the following we make this construction formal, and show that this $gr(g)$ does not satisfy Proposition \ref{homocohomomap} by writing down an explicit CW-complex structure and computing the homology using the cellular chain complex.

\subsubsection{Extending by a Linear Homotopy:}\label{linhom}
Fix a path $P(t)$ in $\DD^2$ given by an embedding $P:[0,1]\ra \DD^2$ (for our purposes $P$ will be an arc of a circle, for instance the one indicated by a 
dashed line in the picture above). Define $f_P:\DD^2 \rightarrow Sub_3(\DD^2)$ as follows:\\   
On the ray from (0,0) to (1/2,0) define $f_P(t,0) = \lbrace P(0),P(2t) \rbrace, t \in [0, \frac{1}{2}]$, so that $f_P(0,0)=\lbrace P(0) \rbrace$. Now define $f_P$ on the circle of radius $t$ by 
$$f_P(tcos 2\pi a, t sin 2\pi a)= \lbrace P(0),P(2t), P(2at) \rbrace$$ 
We extend $f_P$ to $\DD^2$ by using a linear (outward) homotopy from the map $f_P$ on the circle of radius $\frac{1}{2}$ to the constant function $(0,0)$ on the circle of radius $1$, and putting this homotopy as the function on the outward rays. The extended function is also called $f_P$.
 Then $f_P$ is continuous, and there is no lift to $C_3(\DD^2)$ as any assignment of weights to the points forces the weight of the points $P(at)$ for $a\neq 0$ over $( tcos 2\pi a, t sin 2\pi a)$ to be zero (see argument above). 

Define a submultifunction $s(f)$ of the multifunction $f:\DD^n \to Sub_k(\DD^m)$ to be a function $s(f): \DD^n\to Sub_l(\DD^m)$ such that for all $x\in \DD^n$, $s(f)(x)\subset f(x)$. Note that in the above example one can pass to a continuous submultifunction $s(f)$ (leaving out the subset described by the middle strand in the picture) on which  weights can be defined. Hence a homological selection continues to exist
and in particular the conclusion of Proposition \ref{homocohomomap} (ensuring non-trivial maps in homology and cohomology)
 remain true: the submultifunction $s(f)$ defines a class  $\Lambda\in H_2(gr(s(f),\partial gr(s(f))$ so that $\pi_*(\Lambda)\neq 0$ and this $\Lambda$ is also supported on $(gr(f),\partial gr(f))$ as $gr(s(f))\subset gr(f)$.

\subsubsection{CW Complex Structure on $gr(f_P)$}
The topological space $gr(f_P)$ is the union of\\
a) a sphere $S^2$ with the north pole and south pole identified, and \\
b) a disk whose boundary is the wedge of the longitude and the meridian. \\

As a consequence we can write down a CW complex structure of $gr(f_P)$ (recall $\pi: gr(f_P)\rightarrow \DD^2$ is the projection onto the domain of the multifunction $f_P$). There are two 0-cells $v, \, w$ with $v$ as the pinched point (north pole = south pole = the point lying over $(0,0)$, that is, $\pi(v)=(0,0)$) and $w$ as the point on the boundary lying over the point $(1,0)$ ($\pi(w)=(1,0)$). Let $\gamma$ denote the straight line joining $(0,0)$ to $(1,0)$. There are three 1-cells $a_1,~ a_2,~\alpha$, where $\pi^{-1}(\gamma) = a_1\cup a_2 $, and  $\alpha=\pi^{-1}(\partial \DD^2)$. Note that $\DD^2-\gamma - \partial(\DD^2)$ consists only of regular points and hence $\pi^{-1}(\DD^2 - \gamma - \partial \DD^2)= e_1 \cup e_2 \cup e_3$ where $e_i$ are copies of the open $2$-disk. Their closures (also denoted by $e_i$) are three 2-cells. The boundaries in the cellular chain complex are given by: 
$$\partial_2 e_1 = \alpha + a_1 -a_1=\alpha,$$
$$ \partial_2 e_2 = \alpha + a_2 -a_2=\alpha,$$
$$ \partial_2 e_3 = \alpha + a_1 + a_2$$
and
$$\partial_1 a_1 = v - w ,$$
$$\partial_1 a_2=w-v$$ 
$$ \partial_1 \alpha = v- v=0.$$

\subsubsection{From $gr(f_P)$ to  $gr(g_C)$:}
The disk (centered at $(0,0)$) of radius $1/2$ in $\DD^2$ shall be referred to as the {\it half disk}.
We use the above $f_P$ {\it defined on the  half disk} to get a counterexample by wedging three similar maps at a regular point as stated in the proof-idea 
in Section \ref{scheme} (the regular point is labelled $a$  in the diagram below). Formally, we shall use three arcs of a circle as our paths. Fix any circle $C$ in the interior of the range $\DD^2$ (range of the multifunction $\DD^2\to Sub_3(\DD^2)$ to be defined). Define $g_C : \DD^2 \rightarrow Sub_3(\DD^2)$ as follows. Divide $C$ into three arcs $I_1,\, I_2,\, I_3$ by choosing three  points $A_1,A_2,A_3$ in order on it and defining $I_1=A_2A_3,~ I_2=A_3A_1,~ I_3=A_1A_2$. Define $P_i=C-I_i$ for $i=1,2,3$, therefore $P_i$ is a curve joining the two points $A_j$ ($j\ne i$) and passing through $A_i$ (see figure above). The orientation of $P_i$ is defined by starting at the end of $I_i$ and ending at the beginning (for example, $P_1$ starts at $A_3$ passes through $A_1$ and ends at $A_2$). In $\DD^2$ consider three smaller (topological embedded) disks $D_i \, (1\leq i \leq 3)$ in the interior wedged at the point $(0,0)$. One may think of this as a bouquet of petals as in the diagram below meeting at the point $a=(0,0)$. Fix an orientation preserving homeomorphism $\phi_i$ between $D_i$ and the half disk so that $f_{P_i}(\phi_i(0,0)) = \{A_1,A_2,A_3\}$. On the disk $D_i$ define $g_C$ as the function $f_{P_i}
\circ \phi_i $. On the complement of $\cup D_i$ we use a linear homotopy on $g_C$ to get 
the value $g_C(p)= (0,0)$ for $p\in \partial \DD^2$. We prove below that the boundary circle is a free generator for $H_1(gr(g_C); \mathbb{Z})$ thus violating the conclusion of Proposition \ref{homocohomomap}.

\medskip

The topological space $gr(g_C)$ can be obtained as an identification space from the spaces $G_i=gr(f_{P_i})$, $i=1,2,3$ (all the three $G_i$ are homeomorphic) as follows: \\
The fiber  over a regular point $a$ of  the disk $D_i$ consists of three inverse images (labeled $A_1,A_2,A_3$ in the diagram below).
The three sets 
of three inverse images of $a$ (corresponding to the three disks $D_1, D_2, D_3$)
are glued together via  cyclic permutations as indicated in  the diagram below.

\begin{center}
\includegraphics[height=90mm]{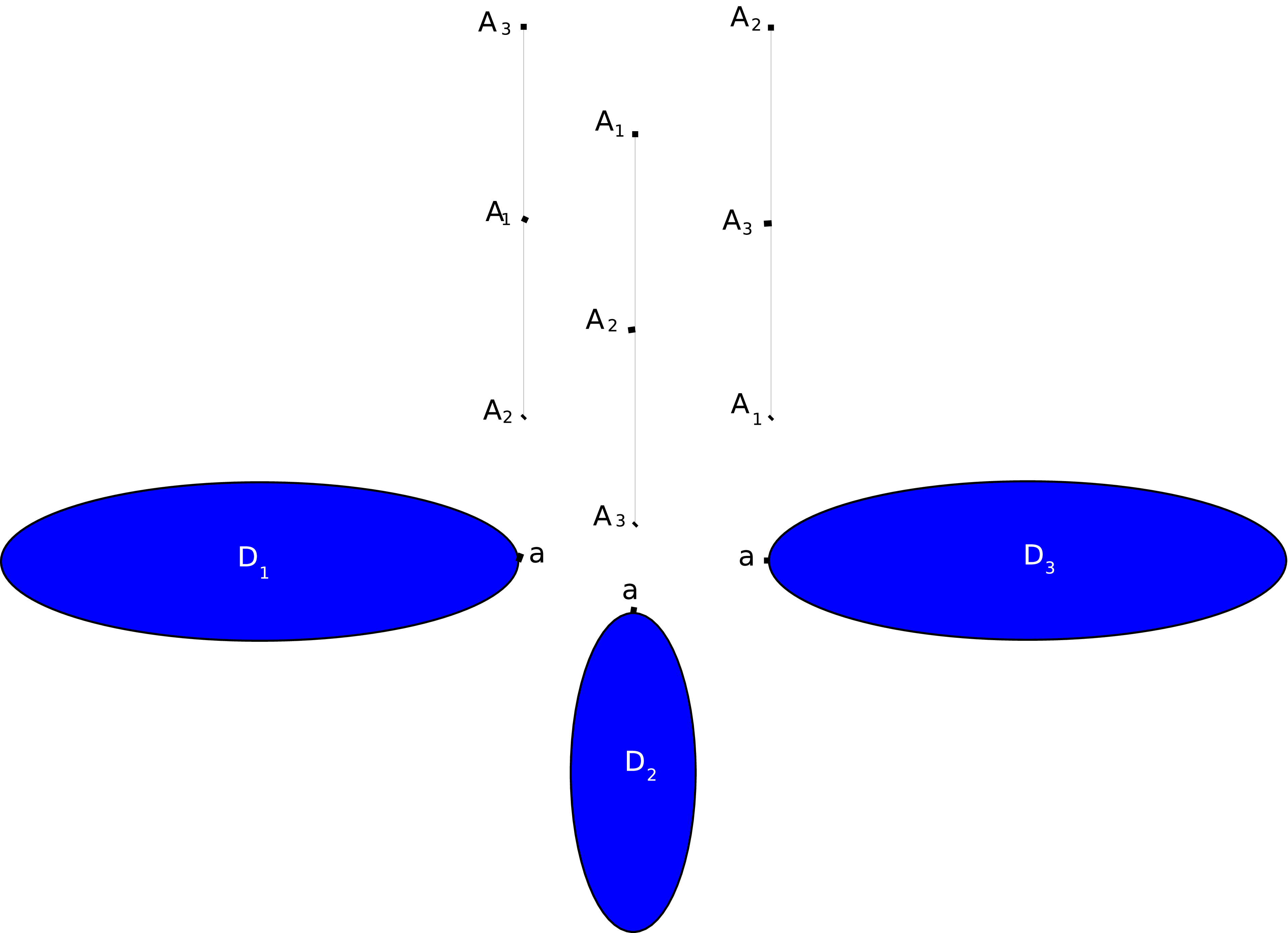}
\end{center} 

\subsubsection{CW Complex Structure on $gr(g_C)$}
From the above, we can write down a CW complex structure on $gr(g_C)$ as follows: \\
There are six 0-cells $v_i,\,w_i \, (i=1,2,3)$ corresponding to the notation in the case of $G_i$; There are nine 1-cells $a_1^i,\, a_2^i \,(i=1,2,3)$ (the same definition as before in the $G_i$) and $\alpha_{1,2},\, \alpha_{2,3}, \, \alpha_{3,1}$ where $\alpha_{i,j}$ joins $v_i$ to $v_j$ on the boundary; There are still three 2-cells (since connecting the disks only results in a bigger disk) $e_1,\,e_2,\,e_3$. 

\subsubsection{Cohomology of $gr(g_C)$} The boundary maps in cellular chain complex are given by 
the following (where   $\alpha_{1,2}+ \alpha_{2,3} + \alpha_{3,1}$ is denoted as $\alpha$ in the cellular chain complex):
 $$\partial_2 e_1 = \alpha + a_1^1 - a_1^1 + a_2^2 -a_2^2 + a_1^3 + a_2^3=\alpha+a_1^3+a_2^3,$$
 $$ \partial_2 e_2 = \alpha+a_1^1 + a_2^1 + a_1^2 - a_1^2 + a_2^3 - a_2^3=\alpha +a_1^1+a_2^1,$$ 
 $$\partial_2 e_3 =  \alpha+a_2^1 -a_2^1 + a_1^2 + a_2^2 + a_1^3 - a_1^3=\alpha + a_2^2+a_1^2.$$
and 
 $$\partial_1 a_j^i = (-1)^j (v^i - w^i) \,(j=1,2, \, i=1,2,3),$$
 $$ \partial_1 \alpha_{k,l} = v_k - v_l \,(k=1,2,3 , \, l=1,2,3).$$  
Now we can compute the CW homology of $gr(g_C)$. We obtain that $H_1(gr(g_C)) \cong \mathbb{Z}$ with generator $\alpha$ which is homologous to each 
 $a_1^i + a_2^i$. Note that $\alpha$ is also the generator of $H_1(\partial gr(g_C))\cong \Z$. Therefore, the image $\partial: H_2(gr(g_C),\partial gr(g_C))\to H_1(\partial gr(g_C))$ is $0$. We have the commutative diagram 
$$\xymatrix{  H_2(gr(g_C),\partial gr(g_C)) \ar[d]^{\pi_*} \ar[rr]^{\partial}  & & H_1(\partial gr(g_C)) \ar[d]^{\cong}\\ 
H_2(\DD^2,\partial \DD^2)  \ar[rr]^{\partial}_{\cong}  & & H_1(\partial \DD^2)  }$$
It follows that for any class $\Lambda\in H_2(gr(g_C),\partial gr(g_C))$, $\pi_*(\Lambda) \in H_2(\DD^2,\partial \DD^2)$ is null-homologous violating the conclusion of Proposition \ref{homocohomomap}, proving the counterexample.  $\Box$

\subsection{Counterexample to Theorem \ref{existNash}}
Theorem \ref{existNash} 
asserts the existence of Nash equilibria under the condition (\textbf{A2}). Using the 
above counter-example, 
we construct an example where  Nash equilibria do not exist. 
To do this we realize the graph
of any continuous function $f:\SM_{-i} \rightarrow Sub_N(\SM_i)$
as the set of minima  of a cost function. 

\begin{prop}\label{minsettoprofit} 
Let $f:\SM_{-i} \rightarrow Sub_N(\SM_i)$ be a continuous function. Then there is a function $F:\SM_{-i} \times \SM_i \rightarrow \mathbb{R}$ such that,
$$f(x)= \lbrace y\,|\, y\, is \, a \, local\, minimum\, of\, y\mapsto F(x,y)\rbrace .$$

\end{prop} 

\begin{proof}
Consider the graph $gr(f)$ of $f$. This is a closed subset of the compact set $\SM_{-i} \times \SM_i$ and 
is therefore, compact. We define $F$ to be the distance function to $gr(f)$ which is 
greater than zero if $(x,y) \notin gr(f)$ and $0$ if $(x,y) \in gr(f)$. Thus, the set of minima 
 of $F$ are exactly the elements of the set $gr(f)$. 
\end{proof}

The above proof also works if we replace local minima
by global minima.
 We now write down
the counterexample to the statement  analogous to Theorem \ref{existNash} obtained by
replacing assumption (\textbf{A2}) by (\textbf{A1}). For the example below,
the number of players $N =2$, $\SM_1$ and $\SM_2$ are two dimensional disks. We have to 
construct functions $f_1: \SM_1 \rightarrow Sub_k(\SM_2)$ and $f_2: \SM_2 \rightarrow 
Sub_l(\SM_1)$ such that $gr(f_1)$ and $gr(f_2)$ do not intersect. 

Define $f_1$ as the example $g_C$ above by using the same formula on the disks $D_i \, (i=1,2,3)$. Note that the disk $\DD^2$ deformation retracts to the space $W=D_1\cup D_2\cup D_3$. Let $r$ be the corresponding retraction. For $x\in \DD^2$, define $f_1(x)=g_C(r(x))$. Therefore, $Im(f_1)\subset Sub_3(C)\subset Sub_3(\DD^2)$.  This defines uniquely a continuous function $f_1: \SM_1 \rightarrow Sub_3(\SM_2)$. 

Next, we shall choose $f_2$ judiciously so that $gr(f_2)$ does not intersect $gr(f_1)$. Since $Im(f_1)$ is contained in $Sub_3(C)$, any intersection value must have its second coordinate in the circle $C$. It is therefore enough to make the judicious choice over $C$ and take any extension. Such extensions always exist since the inclusion of the circle $C$ is a neighborhood deformation retract and $Sub_3(\SM_1)$ is contractible.

Divide the circle $C$ into three arcs $\alpha_1,\, \alpha_2,\, \alpha_3$ where $\alpha_1 $
is the arc joining the mid-point of $A_3A_1$ to the
mid-point of $A_1A_2$ and passing through $A_1$ (and similarly $\alpha_2, \alpha_3$). See figure below where $\alpha_1$ is indicated by the dashed line.
Define $f_2$ on the circle $C$ to be the 
{\it single valued} function mapping $\alpha_i$ to $\partial D_i\subset \DD^2$ such that \\
a) the map is an orientation preserving (homeomorphic) embedding in the interior of each $\alpha_i$, \\
b) the boundary points of $\alpha_i$ are both mapped to $(0,0)$, and\\
c) the point $A_i$ on $\alpha_i$ is mapped to the point diametrically opposite to $(0,0)$ in $\partial D_i$.\\

In the figure below, the image of $\alpha_1$ under $f_2$ is indicated by a bold line.

\begin{center}
\includegraphics[height=90mm]{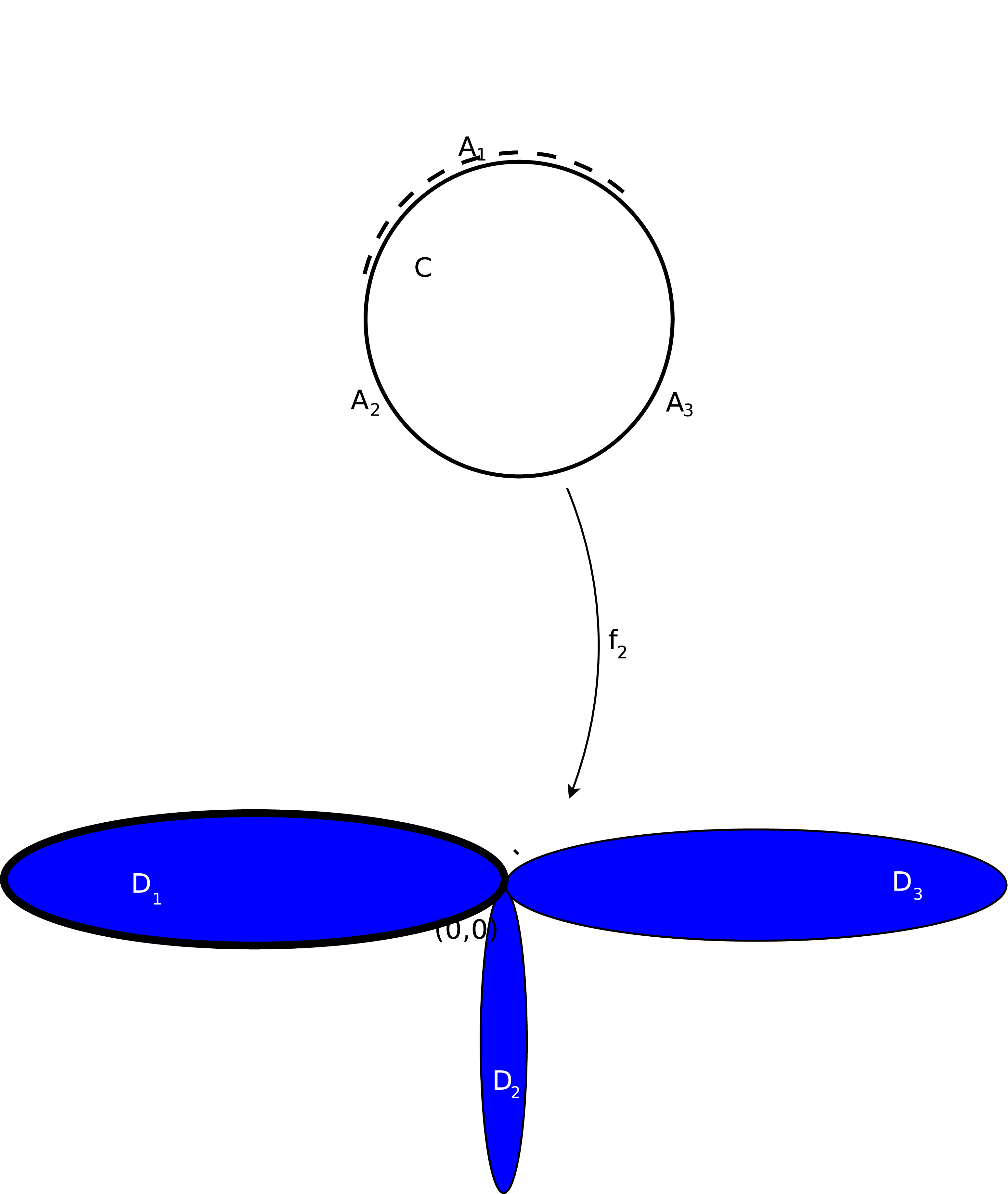}
\end{center}

Recall that
this point on $\partial D_1$ (resp.  $\partial D_2$,  $\partial D_3$) is mapped by $f_1$ to $\lbrace A_2,A_3\rbrace$(resp. $\lbrace A_1,A_3\rbrace$, $\lbrace A_1, A_2\rbrace$).
 We extend $f_2$ to $\SM_2$ by the recipe given in the previous paragraph. 

We wish to prove that $gr(f_1) \cap gr(f_2) = \emptyset$. Note that the contrary would imply that there is an $x \in \SM_2$ with $x\in f_1(f_2(x))$ because $f_2$ is single valued on $C$ and $Im(f_1)\subset Sub_3(C)$. The point $x$ cannot belong to the boundary of the $\alpha_i$ because then $f_1(f_2(x))=\lbrace A_1,A_2,A_3\rbrace$ and none of the $A_j$ lie on the boundary of any $\alpha_i$. 

If $x$ is in the interior of $\alpha_1-\{A_1\}$ then $f_1(f_2(x))$ has the form $\lbrace A_2, A_3, y(x)\rbrace$, where $y(x)$ is a function of $x$ that is continuous except at the point $A_1\in \alpha_1$. As the point $x$ traverses from the mid-point of $A_3A_1$ towards $A_1$ the point $y(x)$ starts at $A_1$ and moves along the arc $A_1A_2$ towards $A_2$. We have $f_1(f_2(A_1))=\{A_2,A_3\}$. As $x$ move from $A_1$ towards the midpoint of $A_1A_2$, the point $y(x)$ starts close to $A_3$ and moves towards $A_1$.  Therefore  if $x\in A_3A_1 \cap \alpha_1$, $y(x)\in A_1A_2$ and if $x\in A_1 A_2\cap \alpha_1$, then $y(x)\in A_3A_1$. Therefore, $x\notin f_1(f_2(x))$. 

The same proof works for $x$ in the interior of $\alpha_2$ and $\alpha_3$. Therefore, $f_1$, $f_2$ as above give cost functions $F_1$, $F_2$ by Proposition \ref{minsettoprofit} such that the game defined by the 
 cost functions $F_i$ for player $i$ ($i=1,2$) does not have a Nash equilibrium.   {\it  Thus Theorem \ref{existNash}
does not hold if    assumption (\textbf{A2}) is replaced by (\textbf{A1}).} This completes the counterexample. $\Box$

\medskip

\noindent {\bf Acknowledgments:} The third author would like to thank Sasha Dranishnikov for telling us the proof
idea of
 Lemma \ref{Dold}.

\end{document}